\documentclass[3p]{elsarticle}

\makeatletter
\def\ps@pprintTitle{
    \let\@oddhead\@empty
    \let\@evenhead\@empty
    \def\@oddfoot{}
    \let\@evenfoot\@oddfoot}
\makeatother

\journal{}

\usepackage{amsmath,amsfonts,amsthm,amscd,amssymb}
\usepackage[utf8]{inputenc}
\newcommand{\R}{\mathbb{R}}
\newcommand{\N}{\mathbb{N}}
\newcommand{\fucik}{Fu\v{c}\'{\i}k}
\newcommand{\norma}[1]{\left\|#1\right\|}
\newcommand{\skals}[2]{\left\langle#1, #2\right\rangle}
\newcommand{\QN}{\mathcal{Q}^{\texttt{n}}}
\newcommand{\QE}{\mathcal{Q}^{\texttt{e}}}
\newcommand{\QW}{\mathcal{Q}^{\texttt{w}}}
\newcommand{\QS}{\mathcal{Q}^{\texttt{s}}}
\newcommand{\AS}{A^{\texttt{s}}}
\newcommand{\adj}{\ast}
\newcommand{\PP}{P}
\newcommand{\PK}{P_{K}}
\newcommand{\dd}{\mathrm{d}}

\DeclareMathOperator{\dom}{Dom}
\DeclareMathOperator{\diag}{diag}
\DeclareMathOperator{\Ker}{Ker}
\DeclareMathOperator{\Img}{Im}
\DeclareMathOperator{\sspan}{span}
\DeclareMathOperator{\codim}{codim}
\DeclareMathOperator{\sgn}{sgn}

\newtheorem{theorem}{Theorem}
\newdefinition{definition}[theorem]{Definition}
\newdefinition{remark}[theorem]{Remark}
\newdefinition{example}[theorem]{Example}
\theoremstyle{definition}
\newtheorem*{acknowledgement}{Acknowledgement}

\begin{document}

\begin{frontmatter}

\title{\textbf{Fu\v{c}\'{\i}k spectrum for discrete systems:\\ curves and their tangent lines}}

\author{Gabriela Holubov\'{a}\corref{cor1}}
\ead{gabriela@kma.zcu.cz}
\cortext[cor1]{Corresponding author}

\author{Petr Ne\v{c}esal}
\ead{pnecesal@kma.zcu.cz}

\address{Department of Mathematics and NTIS, Faculty of Applied Sciences, University of~West Bohemia, Univerzitn\'{\i}~8, 301~00~Plze\v{n}, Czech Republic}

\begin{abstract}
In this paper, we study the \fucik~spectrum of a square matrix $A$ and 
provide necessary and sufficient conditions for the existence of \fucik\ curves 
emanating from the point $(\lambda,\lambda)$ with $\lambda$ being a real eigenvalue of~$A$. 
We extend recent results by Maroncelli (2024) and remove his assumptions on symmetry 
of $A$ and simplicity of $\lambda$. We show that the number of \fucik\ curves can 
significantly exceed the multiplicity of $\lambda$ and determine all the possible 
directions they can emanate in. We also treat the situation when the algebraic 
multiplicity of $\lambda$ is greater than the geometric one and show that in 
such a case the \fucik\ curves can loose their smoothness and provide the slopes 
of their ``one-sided tangent lines''. Finally, we offer two possible generalizations: 
the situation off the diagonal and \fucik\ spectrum of a general Fredholm operator on 
the Hilbert space with a lattice structure.
\end{abstract}


\end{frontmatter}

\section{Introduction}
In \cite{maroncelli}, Maroncelli describes (besides others) the \fucik\ spectrum of a symmetric matrix $A$ 
close to the diagonal. In particular, he shows that there exists a smooth decreasing \fucik\ curve emanating 
from the point $(\lambda,\lambda)$ with the slope $-\norma{u^+}^2/\norma{u^-}^2 $, where $\lambda$ is 
a simple eigenvalue of $A$ and $u$ the corresponding eigenvector with all components nonzero. 
In his proof, the following assumptions are essential:
\begin{description}
\item[$(\texttt{A1})$] symmetry of $A$, 
\item[$(\texttt{A2})$] simplicity of $\lambda$,
\item[$(\texttt{A3})$] nontriviality of all components of $u$.
\end{description}
In concluding remarks he asks, 
what happens if some of these three conditions is violated.

In this paper, we provide the answers and remove assumptions 
$(\texttt{A1})$, $(\texttt{A2})$ and partially also $(\texttt{A3})$. 
Motivated by \cite{bfs1} and \cite{bfs2}, we show what are the possible slopes of \fucik\ curves 
emanating from $(\lambda,\lambda)$ in the case of a non-symmetric matrix and a non-simple real eigenvalue 
$\lambda$ (see Theorems \ref{th_NC1} and \ref{th_NC2}). Under additional nondegeneracy condition and 
assumption $(\texttt{A3})$ we prove the existence of these curves (see Theorem~\ref{th_SC}). Moreover, we also study 
the degenerate case when the algebraic multiplicity of $\lambda$ is higher then the geometric one. 
In such a case,
we show that the bifurcating \fucik\ curves can be non-smooth and provide slopes of 
their one-sided tangent lines (see Section \ref{sec_algebraic}).

All our results are illustrated by a series of examples (Sections \ref{sec_examples}, \ref{sec_examples2}).
Moreover, we suggest a possible generalization to other Hilbert structures and linear 
mappings (see Section \ref{sec_general}). For this reason, we use the symbol 
$H$ for the Euklidean space throughout the text.

Last but not least, let us recall that the complete characterization of the \fucik\ spectrum of 
a matrix belongs to extremely difficult problems even in small dimensions. Pioneering work was done 
in the paper \cite{margulies} by Margulies and Margulies, some partial results can be found in \cite{hn, nl, ps} and 
in the above mentioned paper \cite{maroncelli}, which provides also a nice review of the state of the art.

\section{Preliminaries}
\label{sec_prelim}
Let us consider $H = \R^n$, $n \in \N$, endowed with the standard Euclidean scalar product $\skals{\cdot}{\cdot}$ and Euclidean norm $\norma{\cdot}$.
That is,
$u \in H$ is an $n$-dimensional vector $u = [u_1,\dots, u_n]^t$,
and let us define its positive and negative parts componentwise, i.e., by
$u^{\pm} := [u_1^{\pm}, \dots, u_n^{\pm}]^t$ with
$$
u_i^+ := \max\{u_i, 0\}, \quad u_i^- := \max\{-u_i, 0\}, \quad i = 1, \dots, n.
$$
Clearly, $u = u^+ - u^-$, and $|u| := [|u_1|, \dots, |u_n|]^t = u^+ + u^-$.
Further, let us consider a \emph{linear operator} $A: H \to H$ represented by a real $n$-by-$n$ matrix, and denote $A^{\adj}$ the corresponding adjoint operator (i.e., the transpose of the matrix $A$).

\begin{definition}
The \fucik~spectrum of an operator $A: H \to H$ is a set $\Sigma(A)$ of all pairs
$(\alpha, \beta)\in\R^{2}$, for which the problem
\begin{equation}
\label{e.0}
  A u = \alpha u^{+} - \beta u^{-}
\end{equation}
has a nontrivial solution $u \in H$.
\end{definition}

Since we study the behavior of the \fucik\ spectrum $\Sigma(A)$ close to the diagonal $\alpha = \beta$, it is more convenient to write (\ref{e.0}) in the equivalent form
\begin{equation}
\label{e.d0}
A u = \frac{\alpha - \beta}{2} |u| + \frac{\alpha + \beta}{2} u.
\end{equation}
From now, let us consider an arbitrary but fixed real eigenvalue of $A$, $\lambda \in \sigma(A) \cap \R$, 
and transform $(\alpha,\beta)$ into $(\varepsilon, \eta)$ by
$$
\alpha := \varepsilon (\eta+1) + \lambda, \qquad \beta := \varepsilon (\eta-1) + \lambda,
$$
that is
$$
\varepsilon = \frac{\alpha - \beta}{2}, \qquad \varepsilon \eta= \frac{\alpha +\beta}{2} - \lambda.
$$
Using this transform, the problem (\ref{e.d0}) reads as
\begin{equation}
  \label{e.d1}
  (A - \lambda I) u = \varepsilon (|u| + \eta u)
\end{equation}
and instead of $(\alpha,\beta,u)$ with $\alpha\neq \beta$, we are looking for a triple
$(\varepsilon, \eta, u)$ for $\varepsilon \not= 0$.
Notice that $\varepsilon$ indicates how ``close'' we are to the diagonal $\alpha = \beta = \lambda$, and $\eta$ is related to the ``direction'' from the point $(\lambda,\lambda)$ on the diagonal in $\alpha\beta$ plane.
Namely, we have
\begin{equation}
\label{e.transform}
\eta = \frac{(\alpha - \lambda) + (\beta - \lambda)}{(\alpha - \lambda) - (\beta - \lambda)} \qquad
\mbox{and}
\qquad
\frac{\eta - 1}{\eta + 1} = \frac{\beta - \lambda}{\alpha - \lambda}.
\end{equation}
In particular, we study the existence of the \fucik\ curve described by $(\varepsilon, \eta(\varepsilon), u(\varepsilon))$
and the limit value 
\begin{align*}
\lim_{\varepsilon \to 0} \eta(\varepsilon) =: \eta_0,
\end{align*} 
which provides us the information about the slope of the \fucik\ curve emanating from the point $(\lambda,\lambda)$.

\medskip

To obtain required information, we will use Lyapunov-Schmidt reduction and Implicit Function Theorem. Let us denote
$$
X_{1} := \Ker (A - \lambda  I), \quad  Y_{2} := \Img (A - \lambda  I),
$$
and by $X_{2}$, $Y_{1}$ their orthogonal complements, i.e.,
$$
  H = X_{1} \oplus X_{2} = Y_{1} \oplus Y_{2}.
$$
We recall that
$$
  Y_{1} = \Ker (A^{\adj} - \lambda  I).
$$
Further, we denote by $P$ and $Q$ the orthogonal projections onto $X_{1}$ and $Y_{1}$, respectively. 
For any $u = u(\varepsilon) \in H$, we have
\begin{eqnarray*}
       Q (A - \lambda  I)  u & = & 0,\\
  (I- Q) (A - \lambda  I)  u & = & (A - \lambda  I)  u.
\end{eqnarray*}
Hence, the problem (\ref{e.d1}) is equivalent to the couple of equations
\begin{eqnarray}
                    0 & = & \varepsilon Q (|u| + \eta u),       \label{e.d2}\\
  (A - \lambda  I)  u & = & \varepsilon (I - Q) (|u| + \eta  u). \label{e.d3}
\end{eqnarray}
Moreover,
$$
  (A - \lambda I) u = (A - \lambda I)(P u + (I - P) u) = (A - \lambda I)(I - P) u
$$
and the restriction
$$
 (A - \lambda I) \big|_{X_{2}}:\ X_{2} \to Y_{2}
$$
is invertible with the inverse $(A - \lambda I)\big|_{X_{2}}^{-1}$ defined on $Y_{2}$.
Hence, for any $\varepsilon \not= 0$, the problem (\ref{e.d2})--(\ref{e.d3}) becomes
\begin{eqnarray}
  0 & = &  Q (|u| + \eta u),  \label{e.d4}\\
  (I - P) u & = & \varepsilon (A - \lambda I) \big|_{X_{2}}^{-1} (I - Q)(|u| + \eta u).
  \label{e.d5}
\end{eqnarray}

\bigskip
Finally, let us recall the following differentiability notions and properties in Euclidean spaces. Let $F: \Omega \subset \R^{k_1} \times \dots \times \R^{k_n} \to \R^m$. 
By a (partial) derivative of $F$ with respect to a vector $x_i = [x_{i,1}, \dots, x_{i,k_i}]^t$ at $x  \in \Omega$ we mean a Jacobi ($m\times k_i$) matrix
$$
\frac{\partial F}{\partial x_i}(x) = \left[ 
\begin{array}{ccc} \displaystyle
\frac{\partial F_1}{\partial x_{i,1}}(x) & \dots & \displaystyle \frac{\partial F_1}{\partial x_{i,k_i}}(x)\\
\\
\displaystyle \frac{\partial F_m}{\partial x_{i,1}}(x) & \dots & \displaystyle \frac{\partial F_m}{\partial x_{i,k_i}}(x)
\end{array}
\right].
$$
Further, let us introduce the ``sign'' mapping $\Xi: \R^n \to \R^{n \times n}$ and
the ``zero-characteristic'' mapping $\Xi^0: \R^n \to \R^{n \times n}$
by diagonal matrices
\begin{eqnarray*}
\Xi(u) & := & \diag\left([\sgn(u_1),\sgn(u_2), \dots, \sgn(u_n)]\right),\\[3mm]
(\Xi^0(u))_{ij} & := & \left\{
\begin{array}{ll}
1 & i=j \mbox{ and } u_i =0,\\
0 & \mbox{otherwise}.
\end{array}
\right. 
\end{eqnarray*}
Notice that $\Xi^0(u) = I - \Xi(|u|)$. If $u_0 \in \R^n$ is such that $u_{0,i} \not= 0$ for all $i=1,\dots,n$, then the mapping $f: u \mapsto |u|$ is differentiable at $u_0$ and the derivative
$$
\frac{\partial f}{\partial u} (u_0) = \Xi(u_0).
$$
Moreover, for any $v,z \in \R^n$ and $\varepsilon > 0$ small enough, we have
$$
f(v+\varepsilon z) = |v + \varepsilon z| = |v| + \varepsilon \, \Xi(v)z + \varepsilon \, \Xi^0(v)|z|.
$$
Now, we are ready to formulate our main results.


\section{Main results}
\label{sec_main}

We will start with a necessary condition, which provides all possible directions the \fucik\ curves can emanate in from the diagonal point $(\lambda, \lambda)$ in $\alpha\beta$-plane.

\begin{theorem}
\label{th_NC1}
Let $A: ~H \to H$ be a linear operator represented by a $n$-by-$n$ matrix with a real eigenvalue $\lambda$.
If there exists a sequence of triples $(\varepsilon_n, \eta_n, u_n) \subset \R \times \R \times H$ solving the \fucik\ spectrum problem \eqref{e.d1} such that $\varepsilon_n \not= 0$, $\|u_n\| = 1$,  $n \in \N$, and $(\varepsilon_n, \eta_n, u_n)  \to (0, \eta_0, u_0)$ for $n \to +\infty$, then necessarily $u_0 \in \Ker(A - \lambda I)$, $u_0 \not= 0$, and the projection of
$(|u_0| + \eta_0 u_0)$ onto $\Ker (A^{\adj} - \lambda I)$ is zero, i.e.
\begin{equation}
  \label{e.NC1}
  \forall  v \in \Ker (A^{\ast} - \lambda  I):\ \skals{(|u_0| + \eta_0 u_0)}{v} = 0.
\end{equation}
\end{theorem}

\begin{proof}
Since the problem \eqref{e.d1} is equivalent to (\ref{e.d4})--(\ref{e.d5}), we get for any $n \in \N$
\begin{eqnarray*}
  0 & = &  Q (|u_n| + \eta_n u_n), \\
  (I - P) u_n & = & \varepsilon_n (A - \lambda I) \big|_{X_{2}}^{-1} (I - Q)(|u_n| + \eta_n u_n).
\end{eqnarray*}
Letting $n \to +\infty$, we obtain (due to the boundedness of $u_n$)
$$
 u_0 = P  u_0, \quad \|u_0\| = 1  \quad \mbox{and} \quad Q (| u_0| + \eta_0  u_0) = 0,
$$
which is the required assertion.
\end{proof}

As a consequence, we obtain an analogous assertion concerting a continuous \fucik\ curve and its slope.

\begin{theorem}
\label{th_NC2}
Let $A: ~H \to H$ be a linear operator represented by a $n$-by-$n$ matrix with a real eigenvalue $\lambda$.
If there exists a continuous \fucik\ curve $(\alpha(t),\beta(t))$ on some $I \ni 0$ with the corresponding \fucik\ eigenvectors $u(t) \in H$ such that $\alpha(t) \not= \beta(t)$ for $t\not= 0$, $(\alpha(0), \beta(0)) = (\lambda,\lambda)$ and $u(0) = u_0$, then necessarily $u_0 \in \Ker(A - \lambda I)$, $u_0 \not= 0$, and 
\begin{equation}
  \label{e.NC2}
  \forall  v \in \Ker (A^{\ast} - \lambda  I):\  
  {\alpha'(0)}{\skals{u_0^+}{v}} = {\beta'(0)}{\skals{u_0^-}{v}}.
\end{equation}
\end{theorem}
\begin{proof}
The assertion follows directly from Theorem \ref{th_NC1} and relations \eqref{e.transform}.
\end{proof}

\begin{remark}
If $A$ is selfadjoint (i.e., represented by a symmetric matrix) and $\lambda$ is simple, then (\ref{e.NC2}) reads as
$$
\frac{\dd \beta}{\dd \alpha} (0) = \frac{\beta'(0)}{\alpha'(0)} = \frac{\skals{u_0^+}{u_0}}{\skals{u_0^-}{u_0}} = - \frac{\norma{u_0^+}^2}{\norma{u_0^-}^2}
$$
which fits the result of Maroncelli \cite{maroncelli}.
\end{remark}

Theorems \ref{th_NC1}, \ref{th_NC2} say, in other words, that if there exists $\eta_0  \in \R$  and $u_0 \in \Ker(A - \lambda I)$, $u_0 \not= 0$, satisfying \eqref{e.NC1},
then there is a possible \fucik\ curve emanating from $(\lambda, \lambda)$ in $\alpha\beta$-plane with the slope 
\begin{equation}
  \label{e.d7old}
 \frac{\dd \beta}{\dd \alpha} = \frac{\eta_0 - 1}{\eta_0 + 1}.
\end{equation}

In particular, let $\lambda$ be an eigenvalue of $A$, 
$p := \dim \Ker (A - \lambda I) = \dim \Ker (A^{\ast} - \lambda I)$ and 
$\Ker (A - \lambda I) = \sspan\{u_1,\dots, u_p\}$, $\Ker (A^{\ast} - \lambda I) = \sspan\{v_1,\dots, v_p\}$. 
To find all the possible slopes of \fucik\ curves (i.e., all possible values of $\eta_0$), 
we can proceed in the following way. 
We take an arbitrary eigenvector $u_{0}$ of $A$ corresponding 
to $\lambda$, i.e., $u_{0} = \sum_{i=1}^{p} c_i u_i$ with $c_i \in \R$, and try to find $\eta_{0}$ and 
$c_i$, $i=1,\dots,p$, such that
\begin{equation}
\label{e.d7}
\|u_{0}\|^2 = 1 \quad \wedge \quad {\skals{|u_{0}| + \eta_{0} u_{0}}{v_j}} = 0, \quad j=1,\dots,p.
\end{equation}
This is a (nonlinear) system of $(p+1)$ equations for $(p+1)$ unknown parameters. 
We will illustrate its solutions on particular examples in the next Section.

\medskip


The existence of the \fucik\ curve is given by the following sufficient condition under stronger assumptions.

\begin{theorem}
  \label{th_SC}
Let $A: ~H \to H$ be a linear operator represented by a $n$-by-$n$ matrix with a real eigenvalue $\lambda$.
  Let $(u_{0}, \eta_{0}) \in H\times\R$, $u_{0} \in \Ker (A - \lambda I)$, $\|u_0\| = 1$, be such that all components of $u_0$ are nonzero and \eqref{e.NC1} holds true.
   Further, let the following nondegeneracy condition be satisfied:
  \begin{equation}
    \tag{\texttt{ND}} \label{ND}
    \left.
    \begin{array}{l}
      c \in \R,\ w \in \Ker (A - \lambda I),\ \skals{u_{0}}{w} = 0, \\[3pt]
      Q (\Xi(u_0) w + \eta_{0} w) + c \, Qu_{0} = 0
    \end{array}
    \right\}
    \implies w=0 \mbox{~and~} c=0.
  \end{equation}
  Then there exist continuous functions $\eta(\cdot)$, $u(\cdot)$ defined in a neighborhood $E$ of $0$ such that $u(0) = u_0$, $\eta(0) = \eta_0$,
and
$$
(A - \lambda I) u(\varepsilon)   = \varepsilon (|u(\varepsilon)| + \eta(\varepsilon) u(\varepsilon)), \quad
                         \|u(\varepsilon)\| = 1, \quad
                             \varepsilon \in E,
$$
i.e., 
$$
(\lambda + \varepsilon (\eta(\varepsilon) +1),\lambda + \varepsilon (\eta(\varepsilon)-1)), \quad \varepsilon \in E,
$$ 
describes a part of a \fucik\ curve of $A$ emanating from $(\lambda,\lambda)$ in $\alpha\beta$ plane.
\end{theorem}

\begin{proof}
From above, for $\varepsilon \not= 0$, the problem (\ref{e.d1}) 
is equivalent to the couple of equations (\ref{e.d4})--(\ref{e.d5}).
Since $X_1$ and $Y_1$ have the same dimension, there exists
an isomorphism $\Lambda$ of $Y_{1}$ onto $X_{1}$.
Hence, instead of (\ref{e.d4})-(\ref{e.d5}), we can write
\begin{eqnarray*}
         P u & = & P u + \Lambda Q (|u| + \eta u),\\
  (I - P)  u & = & \varepsilon (A - \lambda  I) \big|_{X_{2}}^{-1} (I - Q) (|u| + \eta u),
\end{eqnarray*}
or, equivalently,
\begin{equation}
  \label{e.d10}
  u = T(\varepsilon, \eta, u)
\end{equation}
with $T:~\R^2\times H \to H$,
\begin{equation}
  \label{e.d11}
  T(\varepsilon, \eta, u) := \varepsilon (A - \lambda  I) \big|_{X_{2}}^{-1} (I - Q) (|u| + \eta u) + Pu + \Lambda Q (|u| + \eta u).
\end{equation}

\medskip

Now, let us consider the system
\begin{eqnarray}
  \| u \|^2 & = & 1, \label{e.d13}\\
          u & = & T(\varepsilon, \eta,u), \label{e.d12}
\end{eqnarray}
which can be written as $F(\varepsilon, z) = 0$ with $z = (\eta,u) \in \R\times H$
and $F:~ \R\times H \to H \times \R$,
$$
F(\varepsilon,z) = F(\varepsilon,\eta,u) = [\|u\|^2-1;u - T(\varepsilon,\eta,u)]^t,
$$
and apply the Implicit Function Theorem. 
Notice that since $|\cdot|$ is a nonexpansive mapping,
the continuity of the operator $T$ (and hence of $F$) with respect to $z = (\eta,u)$ at $(0, \eta_0, u_{0})$ is ensured.
Further, for $u$ with all components nonzero, we have
\begin{eqnarray*}
    \displaystyle
  \frac{\partial T}{\partial \eta}(\varepsilon, \eta, u)  & = &
  \varepsilon (A - \lambda I) \big|_{X_{2}}^{-1} (I - Q) u + \Lambda Q u,\\
  \frac{\partial T}{\partial u}(\varepsilon, \eta, u)  & = &
  \varepsilon (A - \lambda I) \big|_{X_{2}}^{-1} (I - Q) (\Xi(u) + \eta I) 
 + P  + \Lambda Q (\Xi(u) + \eta I),
\end{eqnarray*}
and hence
\begin{eqnarray*}
    \displaystyle
  \frac{\partial T}{\partial \eta}(0, \eta_0, u_0)  & = &
  \Lambda Q u_0,\\
  \frac{\partial T}{\partial u}(0, \eta_0, u_0)  & = &
P  + \Lambda Q (\Xi(u_0) + \eta_0 I).
\end{eqnarray*}
That is, $\partial F/\partial z (0,z_0)$ is a $(n+1)\times(n+1)$ matrix
$$
\frac{\partial F}{\partial z} (0,z_0) =
\left[
\begin{array}{ll}
0   & 2(u_0)^t\\[2mm]
-T_{\eta}(0,\eta_0,u_0) & I - T_u(0,\eta_0,u_0)
\end{array}
\right] =
\left[
\begin{array}{ll}
0 & 2(u_0)^t\\[2mm]
-\Lambda Q u_0 & I - P  - \Lambda Q (\Xi(u_0) + \eta_0 I)
\end{array}
\right].
$$
The last step is the invertibility of $\frac{\partial F}{\partial z}(0, z_{0})$.
In particular, we must ensure that the system $\frac{\partial F}{\partial z}(0, z_{0})(c,w) = 0 \in \R^{n+1}$, i.e.,
\begin{eqnarray}
 2 \skals{w}{u_{0}} & = & 0, \label{e.d16}\\
  \displaystyle (I - P) w - \Lambda Q (\Xi(u_0) w + \eta_0 w) - c \, \Lambda Q u_{0} & = & 0, \label{e.d17}
\end{eqnarray}
has only a trivial solution $(c, w) = (0, 0)$.
Equations (\ref{e.d16})-(\ref{e.d17}) imply that
$$
  (I - P) w = 0, \quad \skals{w}{u_{0}} = 0,
\quad
  Q (\Xi(u_0) w + \eta_0 w) + c \, Q u_{0} = 0,
$$
which results in $(c, w) = (0, 0)$ due to the nondegeneracy condition (\ref{ND}).
The assertion of Theorem \ref{th_SC} now follows directly from the Implicit Function Theorem.
\end{proof}

\begin{remark}~
\label{rem_selfadjoint}
\begin{enumerate}
\item
If $A$ is a selfadjoint operator (i.e., represented by a symmetric matrix),
the nondegeneracy condition (\ref{ND}) reduces to
\begin{align}  \label{NDS}
    \left.
    \begin{array}{l}
      w \in \Ker (A - \lambda I),\ \skals{u_{0}}{w} = 0, \\[3pt]
      P (\Xi(u_0) w) = \skals{|u_{0}|}{u_{0}} w
    \end{array}
    \right\}
    \implies w=0.
\end{align}
which corresponds to the nondegeneracy condition in \cite[Theorem 3.1]{bfs2}.
Indeed, the selfadjointness of $A$ means that
$A = A^{\adj}$ and
$X_{1} = Y_{1}$, $X_{2} = Y_{2}$, $P = Q$, and the relation (\ref{e.NC1})
implies $\eta_0 = -\skals{|u_{0}|}{u_{0}}$, with $u_{0} \in \Ker(A - \lambda I)$,
$\|u_{0}\| = 1$.
Moreover, $\Xi(u_0) u_{0} = |u_{0}|$.
Hence, if $w \in \Ker(A - \lambda I)$, $\skals{w}{u_{0}} = 0$, then
$$
  \skals{\Xi(u_0) w + \eta_0 w}{u_{0}} = \skals{w}{|u_0| + \eta_{0} u_0} = 0,
$$
which means that $c=0$ in (\ref{ND}).

\item
If $\lambda$ is a simple eigenvalue of a selfadjoint operator $A$, 
the nondegeneracy condition (\ref{ND}) is trivially satisfied.

\item 
Let us note that the nondegeneracy condition \eqref{ND} is not satisfied
if there exists $w \in \Ker (A - \lambda I)$ such that $w$ is perpendicular to $u_{0}$,
all components of $Q (\Xi(u_0) w + \eta_{0} w)$ are zero and 
all components of $Q u_{0}$ are nonzero.

\item 
Finally, the nondegeneracy condition \eqref{ND} is satisfied if for all 
$w \in \Ker (A - \lambda I)$ that are nonzero and perpendicular to $u_{0}$, 
at least one component of $Q (\Xi(u_0) w + \eta_{0} w)$ is nonzero
and at least one component is zero.

\end{enumerate}
\end{remark}

\section{Examples}
\label{sec_examples}

We start with the application of Theorem \ref{th_NC1}. 
The following example shows that the number of (possible) \fucik\ curves 
emanating from the diagonal can significantly exceed the multiplicity of 
the corresponding eigenvalue.

\begin{example}
Let us consider $n\times n$ matrix ($n\in\N$, $n \ge 2$)
\begin{displaymath}
  \AS_{n} = \left[
  \begin{array}{cccc}
    1 & 1 & \dots & 1 \\
    2 & 2 & \dots & 2 \\
    \vdots & \vdots & \dots & \vdots \\
    n & n & \dots & n
  \end{array}\right]
\end{displaymath}
with eigenvalues $\lambda_{1,\dots,n-1} = 0$ and
$\lambda_{n} = n(n+1)/2$.
We have $\dim \Ker \AS_{n}  = \dim \Ker (\AS_{n})^{\ast} = n-1$.
Any $ u \in \Ker \AS_{n}$ can be written as a linear combination of $(n-1)$ eigenvectors
$$
 u = c_{1} u_{1} + c_{2} u_{2} + \dots + c_{n-1} u_{n-1}, \qquad c_{1},\dots,c_{n-1}\in\R,
$$
where $u_{i} = [1, 0, \dots, 0, -1, 0, \dots, 0]^{t}$ has $-1$ on $(i+1)$ position, $i=1,\dots,n-1$.
Moreover, any $v \in \Ker (\AS_{n})^{\ast}$ can be expressed as a linear combination of
eigenvectors $v_{1},\dots,v_{n-1}$ of $(\AS_{n})^{\ast}$,
where $v_{i} = [i+1, 0, \dots, 0, -1, 0, \dots, 0]^{t}$ has $-1$ on $(i+1)$ position, $i=1,\dots,n-1$.

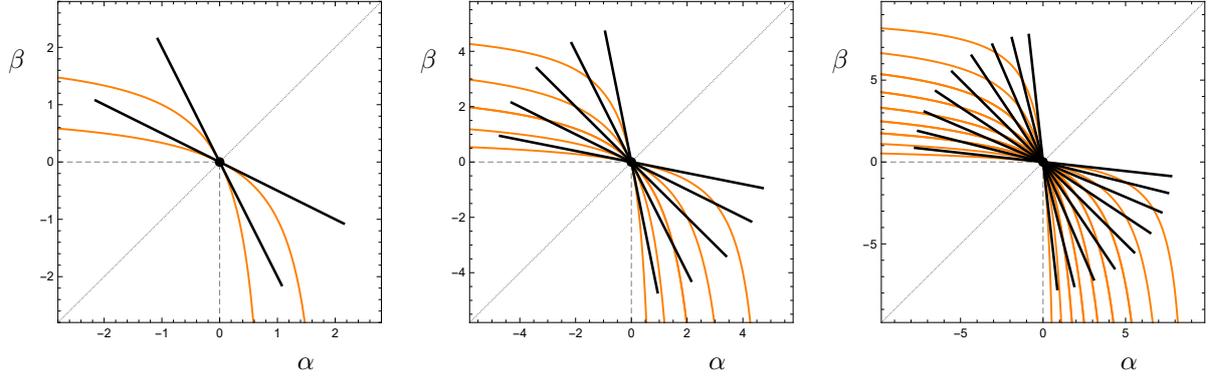
\begin{figure}[t]
  \centerline{
  \setlength{\unitlength}{1mm}
  \begin{picture}(53, 50)(-5, -4)
    \put(  0, 0){\includegraphics[width=4.5cm]{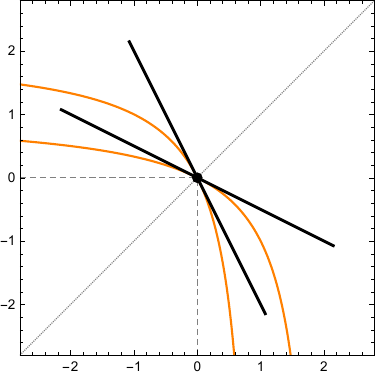}}
    \put(35,-3){\makebox(0,0)[cm]{$\alpha$}}
    \put( -3,37){\makebox(0,0)[cm]{$\beta$}}
  \end{picture}
  \begin{picture}(53, 50)(-5, -4)
    \put(  0, 0){\includegraphics[width=4.5cm]{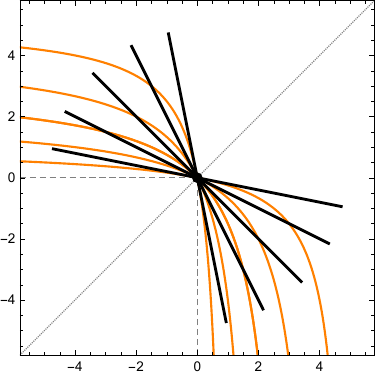}}
    \put(35,-3){\makebox(0,0)[cm]{$\alpha$}}
    \put( -3,37){\makebox(0,0)[cm]{$\beta$}}
  \end{picture}
  \begin{picture}(50, 50)(-5, -4)
    \put(  0, 0){\includegraphics[width=4.5cm]{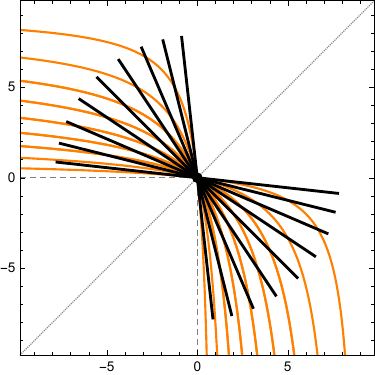}}
    \put(35,-3){\makebox(0,0)[cm]{$\alpha$}}
    \put( -3,37){\makebox(0,0)[cm]{$\beta$}}
  \end{picture}
  }
  \caption{Tangent lines of \fucik~curves emanating from $(0, 0)$, where $0$ is the eigenvalue of
           multiplicity 1 (left,   $2 \times 2$ matrix $\AS_{2}$),
           multiplicity 2 (middle, $3 \times 3$ matrix $\AS_{3}$) and
           multiplicity 3 (right,  $4 \times 4$ matrix $\AS_{4}$).}
  \label{obrtecny}
\end{figure}

We apply Theorem \ref{th_NC1}, i.e., using the necessary condition (\ref{e.NC1}) we find all possible directions the \fucik\ curves can emanate in from the point $(0,0)$ in $\alpha\beta$-plane.
In particular, we are looking for $c_{1},\dots,c_{n-1}\in\R$ and 
$\eta_{0} \in \R$ such that
(\ref{e.d7}) holds true,
which means to solve the following system of $n$ equations
\begin{equation}
  \label{rce2}
  \left.
  \begin{array}{rcl}
    (c_{1} + \dots + c_{n-1})^{2} + c_{1}^{2} + \dots + c_{n-1}^{2} & = & 1, \\
    c_{1} \eta_{0} - |c_{1}| + 2((c_{1} + \dots + c_{n-1})\eta_{0} + |c_{1} + \dots + c_{n-1}|) & = & 0, \\
    c_{2} \eta_{0} - |c_{2}| + 3((c_{1} + \dots + c_{n-1})\eta_{0} + |c_{1} + \dots + c_{n-1}|) & = & 0, \\
    & \vdots & \\
    c_{n-1} \eta_{0} - |c_{n-1}| + n((c_{1} + \dots + c_{n-1})\eta_{0} + |c_{1} + \dots + c_{n-1}|) & = & 0. \\
  \end{array}
  \right\}
\end{equation}
The following table summarizes all possible values of $\eta_{0}$ obtained
from (\ref{rce2}) for $n=2,\ 3,\ 4$, and the slopes 
$\dd \beta / \dd \alpha = (\eta_{0} - 1)/(\eta_{0} + 1)$
of the possible \fucik\ curves emanating from $(0,0)$:

\begin{center}
\begin{tabular}{c|cc}
  $n$ & $\eta_{0}$ & slopes \\[3pt]
\hline\\[-9pt]
   2  & $-\frac{1}{3},  +\frac{1}{3}$
      & $-2, -\frac{1}{2}$ \\[6pt]
   3  & $-\frac{2}{3}, -\frac{1}{3}, 0, +\frac{1}{3}, +\frac{2}{3} $
      & $-5, -2, -1, -\frac{1}{2}, -\frac{1}{5}$ \\[6pt]
   4  & $-\frac{4}{5}, -\frac{3}{5}, -\frac{2}{5}, -\frac{1}{5}, 0, +\frac{1}{5}, +\frac{2}{5}, +\frac{3}{5}, +\frac{4}{5}$
      & $-9, -4, -\frac{7}{3}, -\frac{3}{2}, -1, -\frac{2}{3}, -\frac{3}{7}, -\frac{1}{4}, -\frac{1}{9}$
\end{tabular}
\end{center}
Let us note that for $n = 2$, the multiplicity of the eigenvalue $\lambda = 0$ is $1$
and we have $1$ pair of \fucik\ curves emanating from $(0,0)$ with two different slopes. 
However, for $n=4$, the multiplicity of $\lambda = 0$ is $3$, 
but there are $5$ pairs of \fucik\ curves emanating with $9$ different slopes (two curves coincide). 
\end{example}

The following examples illustrate the application of both 
Theorems \ref{th_NC1} and \ref{th_SC}, i.e., 
both necessary and sufficient conditions for the existence of \fucik\ curves.

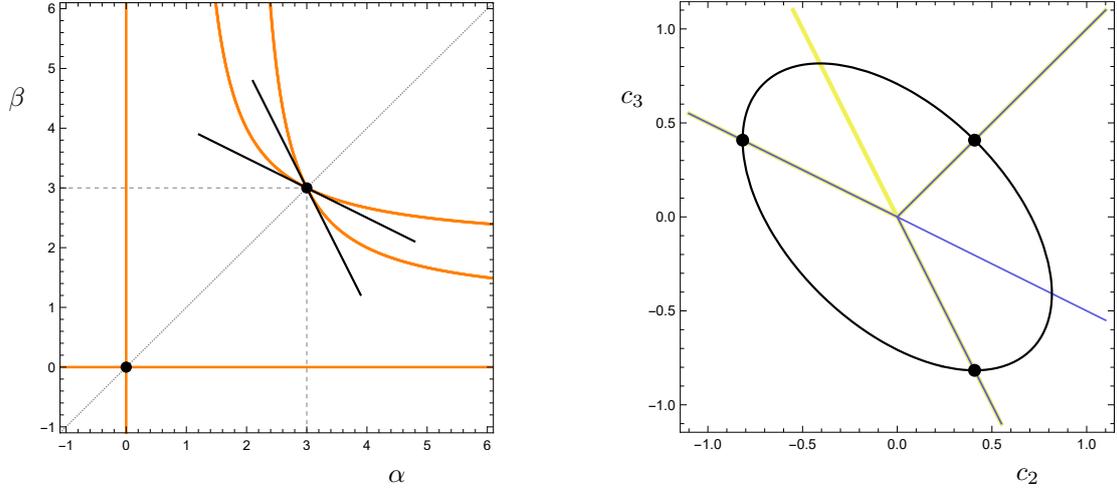
\begin{figure}[t]
  \centerline{
  \setlength{\unitlength}{1mm}
  \begin{picture}(145, 65)(-4, -4)
    \put(  0, 0){\includegraphics[height=6.0cm]{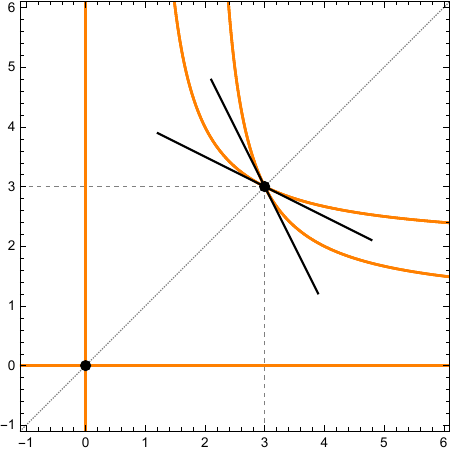}}
    \put( 80, 0){\includegraphics[height=6.0cm]{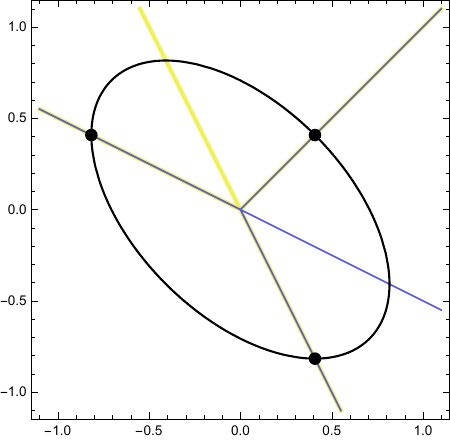}}
    \put( 47,-3){\makebox(0,0)[cm]{$\alpha$}}    
    \put( -3,47){\makebox(0,0)[cm]{$\beta$}}
    \put(130,-3){\makebox(0,0)[cm]{$c_{2}$}}
    \put( 78,47){\makebox(0,0)[cm]{$c_{3}$}}    
  \end{picture}
  }
  \caption{The \fucik~spectrum as orange curves for $3\times 3$ matrix
  $A_{1}$ (left) and its tangent lines (black lines) on the diagonal $\alpha = \beta$.
  Solutions $(c_{2}, c_{3})$ of the nonlinear system \eqref{ns1}-\eqref{ns3}
  for $\eta_{0} = 1/3$ (right) as intersections
  of the ellipse given by \eqref{ns1} and yellow and blue half-lines given by
  \eqref{ns2} and \eqref{ns3}, respectively.}
  \label{obr_01}
\end{figure}

\begin{example} \label{ex1}
Let us consider the symmetric matrix
\begin{displaymath}
  A_{1} = \left[
  \begin{array}{rrr}
     2 & -1 & -1 \\
    -1 &  2 & -1 \\
    -1 & -1 &  2
  \end{array}\right]
\end{displaymath}
with the eigenvalues $\lambda_{1} = 0$, $\lambda_{2} = \lambda_{3} = 3$
and the corresponding eigenvectors
$$
  v_{1} = [1,1,1]^{t}, \qquad
  v_{2} = [-1,0,1]^{t}, \qquad
  v_{3} = [-1,1,0]^{t}.
$$
The \fucik~spectrum of $A_{1}$ consists of four curves 
(see orange curves in Figure \ref{obr_01}, left):
\begin{align*}
  \alpha = 0,\qquad 
  \beta = 0,\qquad 
  \alpha\beta - 2\alpha - \beta = 0,\ \text{where }\alpha > 1,\qquad
  \beta\alpha - 2\beta - \alpha = 0,\ \text{where }\alpha > 2.  
\end{align*}
Now, we examine the situation in the point $(\lambda_2, \lambda_2) = (3,3)$. 
Let us take $u_{0} = c_{2} v_{2} + c_{3} v_{3}$, $c_{2},c_{3}\in\R$, and
solve the nonlinear system in \eqref{e.d7} which has the following form
\begin{align}
  2\left(c_{2}^2 + c_{2}c_{3} + c_{3}^2\right) & = 1, \label{ns1}\\ 
  |c_{2} + c_{3}| - |c_{2}| - \eta_{0}(2 c_{2} +   c_{3}) & = 0, \label{ns2}\\
  |c_{2} + c_{3}| - |c_{3}| - \eta_{0}(  c_{2} + 2 c_{3}) & = 0. \label{ns3}
\end{align}
It is straightforward to verify that this system is solvable if and only if 
$\eta_{0} = \pm 1/3$.
Moreover, all solutions of the system \eqref{ns1}-\eqref{ns3} have the following form
(see Figure \ref{obr_01}, right)
$$
  (\eta_{0}, c_{2}, c_{3}) = \left(\pm\tfrac{1}{3}, \pm\tfrac{1}{\sqrt{6}}, \pm\tfrac{1}{\sqrt{6}}\right),\quad
  (\eta_{0}, c_{2}, c_{3}) = \left(\pm\tfrac{1}{3}, \pm\tfrac{1}{\sqrt{6}}, \mp\tfrac{2}{\sqrt{6}}\right),\quad
  (\eta_{0}, c_{2}, c_{3}) = \left(\pm\tfrac{1}{3}, \mp\tfrac{2}{\sqrt{6}}, \pm\tfrac{1}{\sqrt{6}}\right).
$$
That is, for all these settings the necessary condition \eqref{e.NC1} is satisfied.
Let us take 
$u_{0} = (v_{2}+v_{3})/\sqrt{6} = [-2,1,1]^{t}/\sqrt{6}$
and $\eta_{0} = 1/3$.
We show that the nondegeneracy condition \eqref{ND} is also satisfied.
Indeed, any $w\in \Ker (A_{1} - \lambda_{2} I)$ perpendicular to $u_{0}$ can be written as
\begin{align*}
  w = q\cdot (v_{3} - v_{2}) = [0, q, -q]^{t},\quad q\in\R,
\end{align*}
and moreover, we have $P (\Xi(u_0) w) = P w = w$ and 
  $\skals{|u_{0}|}{u_{0}} w = -w/3$.
Thus, using Theorem \ref{th_SC}, we obtain the existence of 
a \fucik~curve emanating from $(\lambda_{2}, \lambda_{2}) = (3,3)$ with the slope 
$(\eta_{0}-1)/(\eta_{0}+1) = -1/2$.
Similarly, considering $u_{0} = -(v_{2}+v_{3})/\sqrt{6}$ 
and $\eta_0 = -1/3$, we obtain a \fucik~curve emanating from 
$(\lambda_{2}, \lambda_{2}) = (3,3)$ with the slope $-2$. 
The other cases can be treated in the same way with the same results.
\end{example}

\begin{figure}[t]
  \centerline{
  \setlength{\unitlength}{1mm}
  \begin{picture}(145, 65)(-4, -4)
    \put(  0, 0){\includegraphics[height=6.0cm]{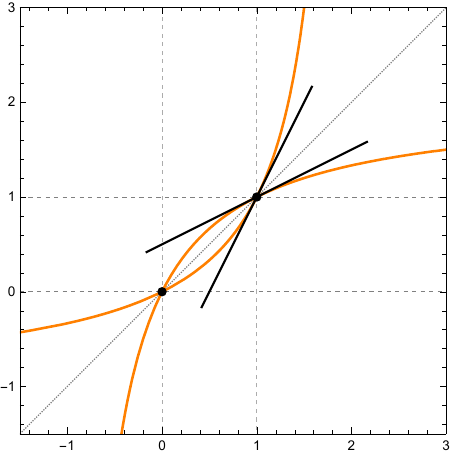}}
    \put( 80, 0){\includegraphics[height=6.0cm]{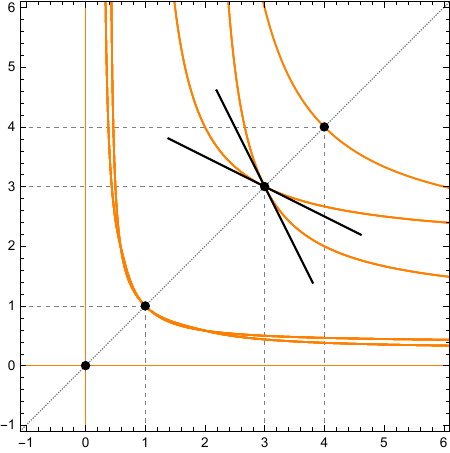}}
    \put( 47,-3){\makebox(0,0)[cm]{$\alpha$}}    
    \put( -3,47){\makebox(0,0)[cm]{$\beta$}}
    \put(130,-3){\makebox(0,0)[cm]{$\alpha$}}
    \put( 78,47){\makebox(0,0)[cm]{$\beta$}}    
  \end{picture}    
  }
  \caption{The \fucik~spectrum as orange curves for $3\times 3$ matrix $A_{2}$ (left)
  and for $6\times 6$ matrix $A_{3}$ (right) and their tangent lines (black lines) 
  on the diagonal $\alpha = \beta$.}  
  \label{obr_02}
\end{figure}
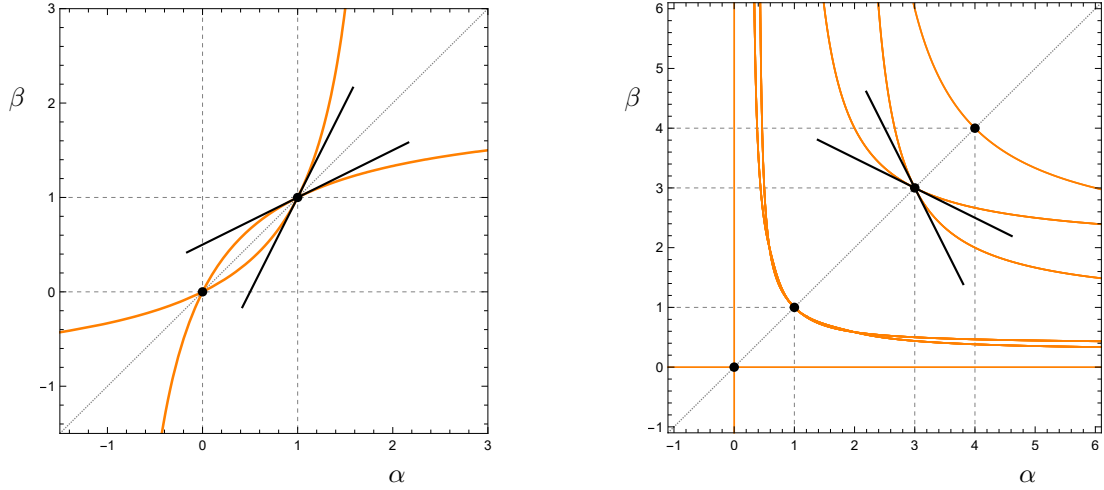

\begin{example} 
Let us consider the nonsymmetric matrix
\begin{displaymath}
  A_{2} = \left[
  \begin{array}{rrr}
     2 &  1 &  1 \\
    -1 &  0 & -1 \\
    -1 & -1 &  0
  \end{array}\right]
\end{displaymath}
with the eigenvalues $\lambda_{1} = 0$, $\lambda_{2} = \lambda_{3} = 1$
and the corresponding right and left eigenvectors
\begin{equation*}
\begin{array}{rcllrcllrcl}
  v_{1} \hspace*{-6pt}&=&\hspace*{-6pt} [-1,+1,+1]^{t}, & \quad &
  v_{2} \hspace*{-6pt}&=&\hspace*{-6pt} [-1,0,+1]^{t}, & \quad &
  v_{3} \hspace*{-6pt}&=&\hspace*{-6pt} [-1,+1,0]^{t},\\[3pt]
  v_{1}^{\ast} \hspace*{-6pt}&=&\hspace*{-6pt} [+1,+1,+1]^{t}, & \quad &
  v_{2}^{\ast} \hspace*{-6pt}&=&\hspace*{-6pt} [+1,0,+1]^{t}, & \quad &
  v_{3}^{\ast} \hspace*{-6pt}&=&\hspace*{-6pt} [+1,+1,0]^{t}.  
\end{array}  
\end{equation*}
(We recall that the left eigenvectors of $A$ are the eigenvectors of the $A^{\adj}$.)
The \fucik~spectrum of $A_{2}$ consists of two curves (see orange curves in Figure \ref{obr_02}, left):
\begin{align*}
  \alpha\beta - 2\alpha + \beta = 0,\ \text{where }\alpha > -1,\qquad
  \beta\alpha - 2\beta + \alpha = 0,\ \text{where }\alpha < 2.  
\end{align*}
Now, let us take 
$u_{0} = (v_{2}+v_{3})/\sqrt{6} = [-2,1,1]^{t}/\sqrt{6}$
and $\eta_{0} = 3$.
At first, the necessary condition \eqref{e.NC1} is satisfied.
Indeed, we have $|u_{0}| + \eta_{0} u_{0} = 4[-1,1,1]^{t}/\sqrt{6}$ and
$\skals{|u_{0}| + \eta_{0} u_{0}}{a v_{2}^{\ast} + b v_{3}^{\ast}} = 0$ for all $a,b\in\R$.
At second, the nondegeneracy condition \eqref{ND} is also satisfied.
Indeed, any $w\in \Ker (A_{2} - \lambda_{2} I)$ perpendicular to $u_{0}$ can be written as
\begin{align*}
  w = q\cdot (v_{3} - v_{2}) = [0, q, -q]^{t},\quad q\in\R,
\end{align*}
and moreover, we have $ Q(\Xi(u_0) w + \eta_{0} w) = Q(w+\eta_{0} w) = (1 + \eta_{0})w$ and 
  $c\cdot Q u_{0} = -c[2,1,1]^{t} /(3\sqrt{6})$.
Thus, using Theorem \ref{th_SC}, we obtain the existence of 
a \fucik~curve emanating from $(\lambda_{2}, \lambda_{2}) = (1,1)$ with the slope 
$(\eta_{0}-1)/(\eta_{0}+1) = 1/2$.
Similarly, considering $u_{0} = -(v_{2}+v_{3})/\sqrt{6}$ and $\eta_0 = -3$, 
we obtain a \fucik~curve emanating from 
$(\lambda_{2}, \lambda_{2}) = (1,1)$ with the slope $2$. 
Moreover, no other $\eta_0$ satisfies \eqref{e.NC1} for $\lambda = \lambda_2$.

\end{example}

\begin{example}
Let us consider the symmetric matrix
\begin{displaymath}
  A_{3} = \left[
  \begin{array}{rrrrrr}
     2 & -1 & 0 & 0 & 0 & -1 \\
    -1 &  2 & -1 & 0 & 0 & 0 \\
    0 & -1 &  2 & -1 & 0 & 0 \\
    0 & 0 & -1 &  2 & -1 & 0 \\ 
    0 & 0 & 0 & -1 &  2 & -1 \\     
    -1 & 0 & 0 & 0 & -1 &  2  \\         
  \end{array}\right]
\end{displaymath}
with the eigenvalues 
$\lambda_{1} = 0$, 
$\lambda_{2} = \lambda_{3} = 1$,
$\lambda_{4} = \lambda_{5} = 3$,
$\lambda_{6} = 4$
and the corresponding eigenvectors
\begin{equation*}
\begin{array}{llllllllllll}  
  v_{1} \hspace*{-6pt}&=&\hspace*{-6pt} [1,1,1,1,1,1]^{t}, &\quad&
  v_{2} \hspace*{-6pt}&=&\hspace*{-6pt} [-1,-1,0,1,1,0]^{t}, &\quad& 
  v_{3} \hspace*{-6pt}&=&\hspace*{-6pt} [1,0,-1,-1,0,1]^{t}, \\[6pt]  
  v_{4} \hspace*{-6pt}&=&\hspace*{-6pt} [-1,0,1,-1,0,1]^{t}, &&
  v_{5} \hspace*{-6pt}&=&\hspace*{-6pt} [-1,1,0,-1,1,0]^{t}, &&
  v_{6} \hspace*{-6pt}&=&\hspace*{-6pt} [-1,1,-1,1,-1,1]^{t}.      
  \end{array}
\end{equation*}
Let us note that $A_3$ corresponds to the discretization 
of the one-dimensional Laplace operator with periodic boundary conditions.
Using Theorem \ref{th_SC} as in Example \ref{ex1}, we obtain the existence of 
the \fucik~curves emanating from $(\lambda_{4}, \lambda_{4}) = (3,3)$ with slopes 
$-1/2$ and $-2$ (see Figure \ref{obr_02}, right).
Now, let us consider the \fucik~curves that emanate from $(\lambda_{2}, \lambda_{2}) = (1,1)$
and let us take 
$u_{0} = (v_{2}-v_{3})/(2\sqrt{3})=[-2,-1,1,2,1,-1]^{t}/(2\sqrt{3})$
and $\eta_{0} = 0$.
Then the necessary condition \eqref{e.NC1} is satisfied since 
$\langle|u_{0}|, a v_{2} + b v_{3}\rangle = 0$ for all $a,b\in\R$.
Unfortunately, the nondegeneracy condition \eqref{ND} having the form of \eqref{NDS} is not satisfied.
Indeed, if we take the nonzero $w\in\Ker(A_{3} - \lambda_{2}I)$ as 
$w = v_{2} + v_{3} = [0,-1,-1,0,1,1]^{t}$ then 
it is perpendicular to $u_{0}$ and 
the equality $P (\Xi(u_0) w) = \skals{|u_{0}|}{u_{0}} w$ in \eqref{NDS}
is also satisfied since $\skals{|u_{0}|}{u_{0}} = 0$ 
and the vector $\Xi(u_0) w = [0,1,-1,0,1,-1]^{t}$ is perpendicular to $\Ker(A_{3} - \lambda_{2}I)$.
\end{example}


\section{Degenerate Cases in Eigenvalues with Higher Algebraic Multiplicity}
\label{sec_algebraic}

The nondegeneracy condition (\ref{ND}) in Theorem \ref{th_SC} includes the requirement $Qu_0 \not= 0$, or, in other words, $u_0 \not\in \Img (A - \lambda I)$. This ensures us that there exists $v \in \Ker (A^{\ast} - \lambda I)$ such that
$\skals{u_0}{v} \not= 0$. Now, let us have a look what happens when $u_0 \in \Ker(A- \lambda I) \cap \Img (A - \lambda I)$, i.e.,
$$
\forall v \in \Ker (A^{\ast} - \lambda I): \quad \skals{u_0}{v} = 0.
$$
Notice that this implies the existence of
$u_g \in H$ such that
$$
(A - \lambda I) u_g = u_0.
$$
That is, $u_g$ is the generalized eigenvector corresponding to $\lambda$ and the algebraic multiplicity of $\lambda$ is greater than its geometric multiplicity!
Now, we distinguish two cases.

\bigskip

{\bf Case 1.} If
\begin{equation}
\label{e.al00}
\exists u_0 \in \Ker (A - \lambda  I) \cap \Img (A - \lambda I), \ \ \exists v \in \Ker (A^{\ast} - \lambda I): \ \ \skals{|u_0|}{v} \not= 0,
\end{equation}
then the condition (\ref{e.NC1}) implies $|\eta_0| = +\infty$ and there is a possible \fucik\ curve emanating from $(\lambda,\lambda)$ with the slope $\dd \beta/ \dd \alpha = 1$.

\bigskip

{\bf Case 2.} If
\begin{equation}
\label{e.al0}
\exists u_0 \in \Ker (A - \lambda  I) \cap \Img (A - \lambda I), \ \ \forall v \in \Ker (A^{\ast} - \lambda I): \ \ \skals{|u_0|}{v} =  0,
\end{equation}
then the process described in Section \ref{sec_main} fails. In such a case, (\ref{e.NC1}) is trivially satisfied and provides no information about $\eta_0$. Thus, we must proceed more carefully.

\begin{theorem}
\label{th_algebraic}
Let $A: ~H \to H$ be a linear operator represented by a $n$-by-$n$ matrix with a real eigenvalue $\lambda$. Let $u_{0} \in \Ker(A - \lambda I) \cap \Img (A - \lambda I)$, $u_{0} \not= 0$, be such that \eqref{e.al0} holds true.
If there exists a \fucik\ curve of $A$ emanating from the eigenvalue $\lambda$ described by $(\varepsilon,\eta(\varepsilon),u(\varepsilon))$ satisfying \eqref{e.d1} with $u_0=u(0)$ and $$ \lim_{\varepsilon \to 0\pm} \eta(\varepsilon) = \eta_{0\pm}, $$
then necessarily
there exist $z_{0\pm} \in \Ker (A - \lambda I)$, $\skals{z_{0\pm}}{u_{0}} = 0$, such that
\begin{equation}
\label{e.al8}
\forall v \in \Ker (A^{\ast} - \lambda I): \ \ 
\skals{\Xi(u_0)u_{1\pm} \pm \Xi^0(u_0)  |u_{1\pm}|}{v} + \eta_{0\pm}
\skals{u_{1\pm}}{v} = 0,
\end{equation}
with 
\begin{equation}
\label{e.al2}
 u_{1\pm} = (A - \lambda  I) \big|_{_{X_{2}}}^{-1} (|u_{0}| + \eta_{0\pm} u_{0} ) + z_{0\pm}.
\end{equation}
\end{theorem}

\begin{proof}
We will perform the proof for $\varepsilon \to 0+$. The case $\varepsilon \to 0-$ can be treated in the same way.
Let us go back to the decomposition (\ref{e.d2})-(\ref{e.d3}), or (\ref{e.d3})-(\ref{e.d4}), respectively, and suppose that
\begin{equation}
\label{e.al0a}
u (\varepsilon) = u_{0} + \varepsilon u_{1} + o(\varepsilon)
\end{equation}
with $u_0 \in \Ker(A- \lambda I) \cap \Img (A - \lambda I)$ satisfying (\ref{e.al0}).
The equation (\ref{e.d3}) is then equivalent to
$$
\varepsilon (A - \lambda I) (u_{1} + o(1)) = \varepsilon (I - Q) (|u(\varepsilon)| + \eta(\varepsilon) u(\varepsilon)).
$$
Dividing by $\varepsilon$ and letting $\varepsilon \to 0+$, we obtain
$$
(A - \lambda I) u_{1} = (I - Q) (|u_{0}| + \eta_{0+} u_{0}) = |u_{0}| + \eta_{0+} u_{0},
$$
since $Q(|u_{0}|) = Q u_{0} = 0$.
Moreover, the restriction $(A - \lambda  I) \big|_{_{X_{2}}}: X_{2} \to Y_{2}$ is invertible and thus
$$
 u_{1} = (A - \lambda  I) \big|_{_{X_{2}}}^{-1} (|u_{0}| + \eta_{0+} u_{0}) + z_{0},
$$
where $z_{0} \in \Ker (A - \lambda  I)$ is such that $\skals{z_{0}}{u_{0}} = 0$. 
(Notice that adding any multiple of $u_{0}$ to $u_{1}$ has no influence due to (\ref{e.al0}).)

For $\varepsilon$ small enough, we have
$$
|u(\varepsilon)| = |u_0 + \varepsilon u_1 + o(\varepsilon))| = |u_0| + \Xi(u_0)(\varepsilon u_1 + o(\varepsilon)) + \Xi^0(u_0) |\varepsilon u_1 + o(\varepsilon)|.
$$
Hence, taking the equation (\ref{e.d4}), i.e.,
$$
Q (|u(\varepsilon)| + \eta(\varepsilon) u(\varepsilon)) = 0,
$$
substituting from above and using $Q(u_0) = Q(|u_0|) = 0$, implies
$$
Q (\Xi(u_0)(\varepsilon u_1 + o(\varepsilon)) + \Xi^0(u_0) |\varepsilon u_1 + o(\varepsilon)|) + \eta(\varepsilon) 
Q (\varepsilon u_1 + o(\varepsilon)) = 0.
$$
Dividing now by $\varepsilon$ and letting $\varepsilon \to 0+$ leads to
$$
Q (\Xi(u_0)u_1 + \Xi^0(u_0) |u_1|) + \eta_{0+}
Q (u_1) = 0,
$$
which is the proved assertion.
\end{proof}

\begin{remark}
Since $(A - \lambda I) \big|_{_{X_{2}}}^{-1} u_{0} = u_{g}$, the relation (\ref{e.al2}) can be simplified into
\begin{equation}
\label{e.al3}
 u_{1\pm} = (A - \lambda I) \big|_{_{X_{2}}}^{-1} |u_{0}| + \eta_{0\pm} u_{g} + z_{0\pm}.
\end{equation}
Moreover,
if $|u_{0}| = u_{0}$, then $(A - \lambda I) \big|_{_{X_{2}}}^{-1} |u_{0}| = u_{g}$ and (\ref{e.al2}) becomes
\begin{equation}
\label{e.al6}
u_{1\pm} = (1 + \eta_{0\pm}) u_{g} + z_{0\pm}.
\end{equation}
Finally, if $\lambda$ is simple, $z_{0\pm} \equiv 0$.
\end{remark}

\begin{figure}[t]
  \centerline{
  \setlength{\unitlength}{1.27mm}
  \begin{picture}(70, 70)(0, 0)
    \put(0, 0){\includegraphics[height=9cm]{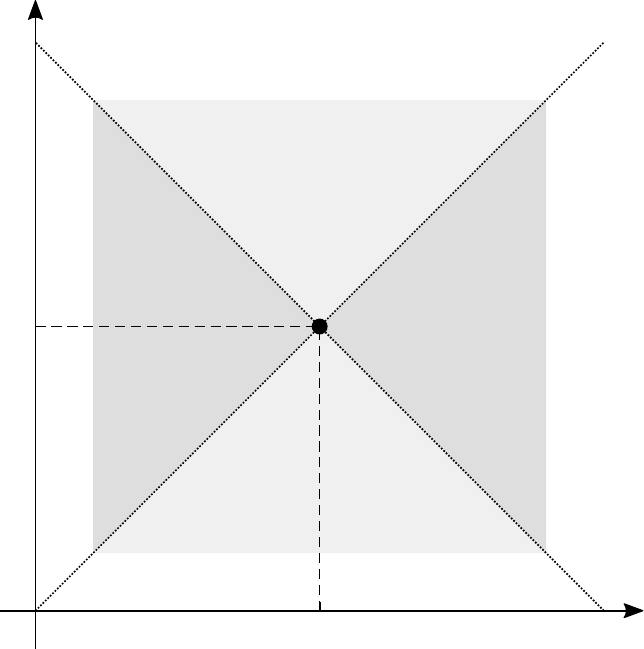}}
    \put(63.5,1){\makebox(0,0)[cm]{$\alpha$}}
    \put(1,63.5){\makebox(0,0)[cm]{$\beta$}}
    \put(56.5,63.5){\makebox(0,0)[cm]{$\beta = \alpha$}}    
    \put(17,63.5){\makebox(0,0)[cm]{$\beta = 2\lambda - \alpha$}}        
    \put(35,1){\makebox(0,0)[cm]{$\lambda$}}
    \put(1,35){\makebox(0,0)[cm]{$\lambda$}}
    \put(49,31){\makebox(0,0)[cm]{$\eta_{0+} \ge 0$}}
    \put(35,53){\makebox(0,0)[cm]{$\eta_{0-} \le 0$}}    
    \put(22,31){\makebox(0,0)[cm]{$\eta_{0-} \ge 0$}}
    \put(42,18){\makebox(0,0)[cm]{$\eta_{0+} \le 0$}}
    \put(22,39){\makebox(0,0)[cm]{$\QW(\lambda)$}}        
    \put(35,47){\makebox(0,0)[cm]{$\QN(\lambda)$}}
    \put(29,18){\makebox(0,0)[cm]{$\QS(\lambda)$}}    
    \put(49,39){\makebox(0,0)[cm]{$\QE(\lambda)$}}                    
  \end{picture}
  }
  \caption{Quadrants determined by signs of $\eta_{0\pm}$.}
  \label{kvadranty}
\end{figure}
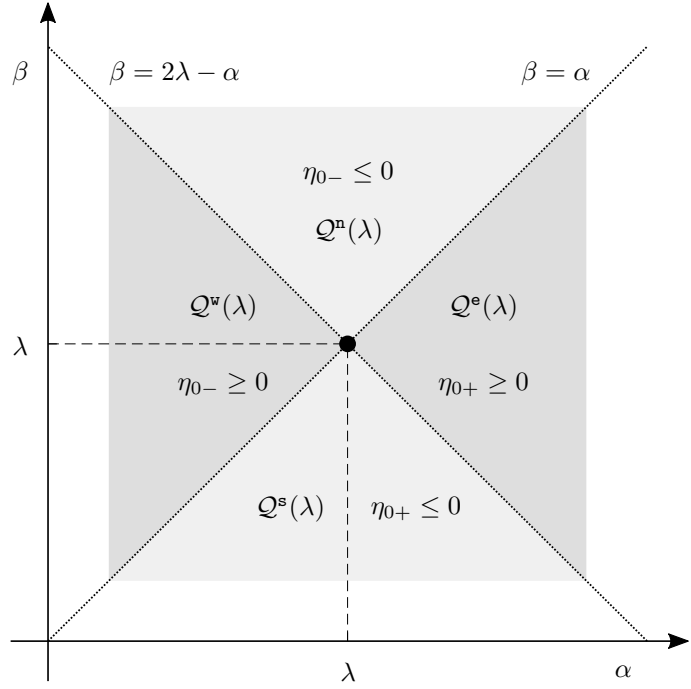

\begin{remark}
\label{rem_kvadranty}
Notice that the sign of $\eta_{0\pm}$ determines the ``direction'' of the possible \fucik\ curve bifurcating from $(\lambda, \lambda)$ with the slope $(\eta_{0\pm} - 1)/(\eta_{0\pm} + 1)$. In particular,
$$
\begin{array}{lclcl}
\eta_{0+} \geq 0 & \Leftrightarrow & 
(\alpha,\beta)\in\QE(\lambda) := 
\left\{(\alpha,\beta)\in\R^{2}:\ (\alpha - \lambda) \geq |\beta - \lambda|\right\}, \\[3pt]
\eta_{0+} \leq 0 & \Leftrightarrow & 
(\alpha,\beta)\in\QS(\lambda) := 
\left\{(\alpha,\beta)\in\R^{2}:\ (\beta - \lambda) \leq -|\alpha - \lambda|\right\}, \\[3pt]
\eta_{0-} \geq 0 & \Leftrightarrow & 
(\alpha,\beta)\in\QW(\lambda) := 
\left\{(\alpha,\beta)\in\R^{2}:\ (\alpha - \lambda) \leq -|\beta - \lambda|\right\}, \\[3pt]
\eta_{0-} \leq 0 & \Leftrightarrow & 
(\alpha,\beta)\in\QN(\lambda) := 
\left\{(\alpha,\beta)\in\R^{2}:\ (\beta - \lambda) \geq |\alpha - \lambda|\right\}. 
\end{array}
$$
Thus, the sign of $\eta_{0\pm}$ determines uniquely the ``quadrant'' in $\alpha\beta$ plane (see Figure \ref{kvadranty}).
It means that if the algebraic multiplicity of the eigenvalue $\lambda$ is higher than the geometric one, the bifurcating \fucik\ curves can be nonsmooth with``one-sided tangent lines'' at $(\lambda,\lambda)$. Notice, that this cannot happen if the geometric and algebraic multiplicities of $\lambda$ coincide.
\end{remark}


\begin{remark}
\label{rem_higher}
Again, relation (\ref{e.al8}) with $u_{1}$ given by (\ref{e.al2}) can degenerate and 
provide no information about $\eta_{0\pm}$ if
$$
\skals{\Xi(u_0)u_{1\pm} \pm \Xi^0(u_0) |u_{1\pm}|}{v} =
\skals{u_{1\pm}}{v} = 0
$$
for any $v \in \Ker (A^{\ast} - \lambda I)$. Notice that this can happen
in some cases when the difference of algebraic and geometric multiplicities of $\lambda$ is 2 or more.
For example, the following matrix 
$$
  A_{4} = \left[
  \begin{array}{rrrr}
    2 &  0 & -2 &  0 \\
    2 & -1 & -2 &  1 \\
    0 &  1 &  0 & -1 \\
   -2 &  3 &  2 &  1
  \end{array}\right]
$$
has the zero eigenvalue with simple geometric multiplicity and higher algebraic multiplicity
such that the relation (\ref{e.al8}) provides no information about $\eta_{0\pm}$.
In such a case, we must use more precise expressions of $u(\varepsilon)$ 
than (\ref{e.al0a}) and include higher order terms. 
\end{remark}


\section{Examples of Degenerate Cases}
\label{sec_examples2}

\begin{example}
\label{ex_degenerate}
Let us consider a matrix
$$
A_{5} = \left[
\begin{array}{rrrr}
3 & -3 & 3 & 0\\
-1 & 6 & -1 & -1\\
1 & -3 & 1 & 4\\
-1 & 6 & -1 & -1
\end{array}
\right]
$$
with the eigenvalues $\lambda_1 = 0$, $\lambda_2 = 3$, $\lambda_3 = 6$.
The eigenvector corresponding to $\lambda_1 = 0$ can be chosen as
$u_0 = [-6, 0, 6, 0]^t$.
The adjoint eigenvector corresponding to $\lambda_1 = 0$ is
$v = [0, 6, 0, -6]^t$
and we have
$\skals{u_0}{v} = \skals{|u_0|}{v} = 0$.
Hence, there must exist a generalized eigenvector $u_g$ to $\lambda_1$, e.g.,
$u_g = [-5, 0, 3, 2]^t$.
Moreover, since $\Ker A_{5}  = \sspan \{u_0\}$, we have $z_{0\pm} \equiv 0$.
That is, $u_{1\pm} = 2 [2,1,2,2]^t/3 + \eta_{0\pm} [-5, 0, 3, 2]^t$ and
the equality (\ref{e.al8}) reads as
\begin{eqnarray*}
1 - |2 + 3\eta_{0+}| + \eta_{0+}(1 - (2+ 3\eta_{0+})) & = & 0,\\
1 - |2 + 3\eta_{0-}| - \eta_{0-}(1 - (2+ 3\eta_{0-})) & = & 0.
\end{eqnarray*}
Hence, we obtain the following results:
\begin{center}
\begin{tabular}{ccc}
  $\eta_{0+}$ &  slope & region \\[3pt]
\hline\\[-9pt]
$-\frac{1}{3}$ & $-2$ & $\QS(0)$ \\[6pt]
$\frac{1}{3}(1-\sqrt{10})$  & $-3 - \sqrt{10}$ & $\QS(0)$ 
\end{tabular}
\hspace{1cm}
\begin{tabular}{ccc}
  $\eta_{0-}$ &  slope & region \\[3pt]
\hline\\[-9pt]
$-\frac{1}{3}$  & $-2$ & $\QN(0)$\\[6pt]
$1$ & $0$ & $\QW(0)$
\end{tabular}
\end{center}
\begin{figure}[h]
  \centerline{
  \setlength{\unitlength}{1mm}
  \begin{picture}(145, 65)(-4, -4)
    \put(  0, 0){\includegraphics[height=6.0cm]{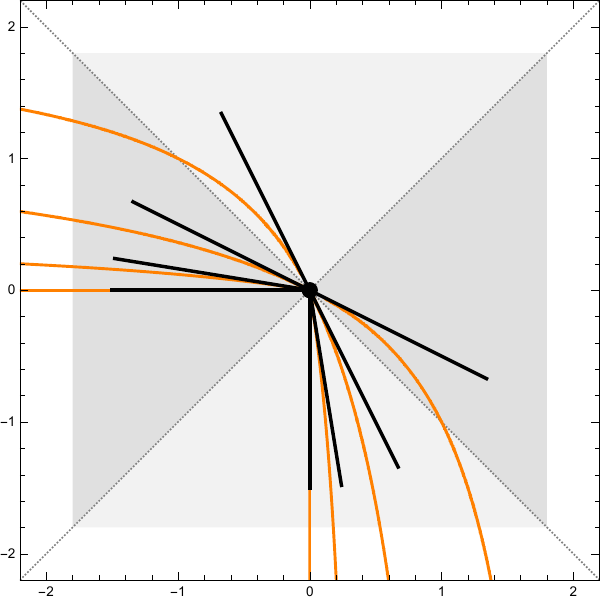}}
    \put( 80, 0){\includegraphics[height=6.0cm]{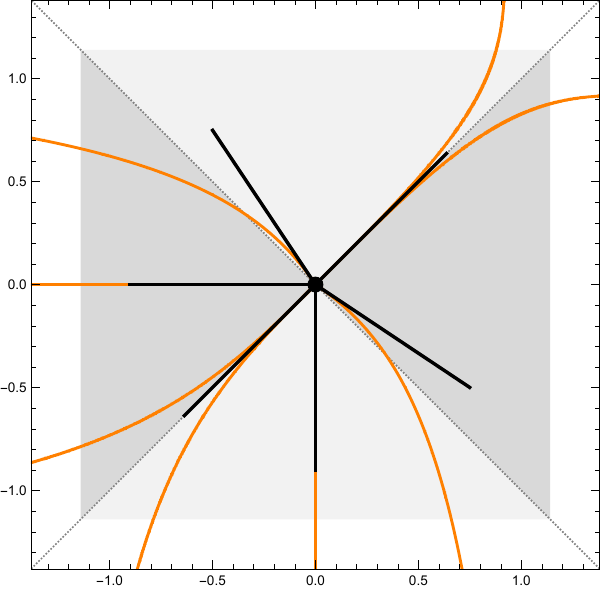}}
    \put( 47,-3){\makebox(0,0)[cm]{$\alpha$}}    
    \put( -3,47){\makebox(0,0)[cm]{$\beta$}}
    \put(130,-3){\makebox(0,0)[cm]{$\alpha$}}
    \put( 78,47){\makebox(0,0)[cm]{$\beta$}} 
    \put(15,25){\makebox(0,0)[cm]{\footnotesize $\QW(0)$}}        
    \put(32,50){\makebox(0,0)[cm]{\footnotesize $\QN(0)$}}
    \put(24,14){\makebox(0,0)[cm]{\footnotesize $\QS(0)$}}    
    \put(48,31){\makebox(0,0)[cm]{\footnotesize $\QE(0)$}}         
    \put(96,25){\makebox(0,0)[cm]{\footnotesize $\QW(0)$}}        
    \put(113,50){\makebox(0,0)[cm]{\footnotesize $\QN(0)$}}
    \put(106,14){\makebox(0,0)[cm]{\footnotesize $\QS(0)$}}    
    \put(129,31){\makebox(0,0)[cm]{\footnotesize $\QE(0)$}}                
  \end{picture}
  }
  \caption{Tangent lines of \fucik~curves emanating from $(0, 0)$, where $0$ 
           is the eigenvalue with higher algebraic multiplicity of $4 \times 4$
           matrices $A_{5}$ (left) and $A_{6}$ (right).}
  \label{fig_degenerate}
\end{figure}
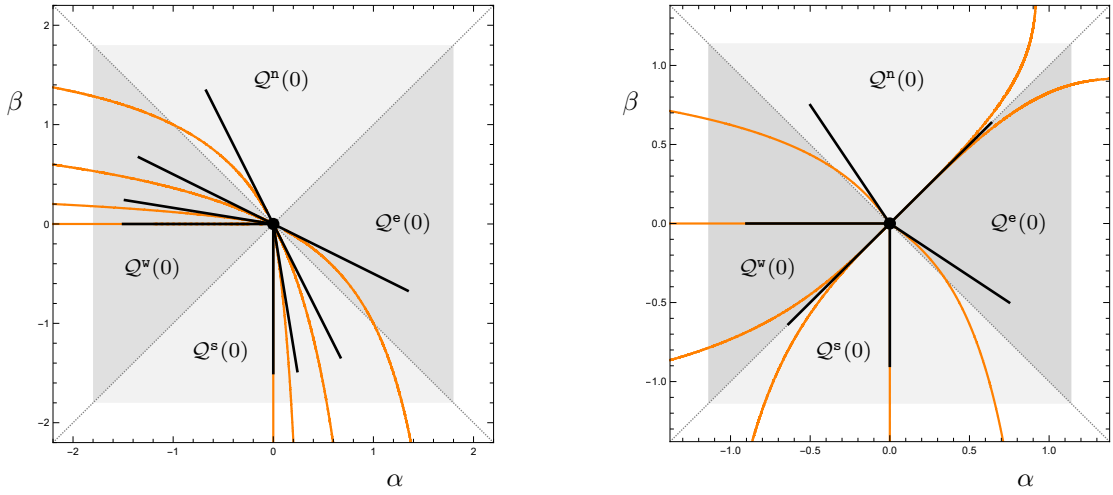
Similarly, for the choice $-u_0$, the equality (\ref{e.al8}) reads as
\begin{eqnarray*}
1 - |2 - 3\eta_{0+}| + \eta_{0+}(1 - (2- 3\eta_{0+})) &=& 0,\\
1 - |2 - 3\eta_{0-}| - \eta_{0-}(1 - (2- 3\eta_{0-})) &=& 0,
\end{eqnarray*}
and we obtain the following results:
\begin{center}
\begin{tabular}{ccc}
$\eta_{0+}$ & slope & region \\[3pt]
\hline\\[-9pt]
$-1$ &  $\infty$ & $\QS(0)$ \\[6pt]
$\frac{1}{3}$  & $-\frac{1}{2}$ & $\QE(0)$ 
\end{tabular}
\hspace{1cm}
\begin{tabular}{ccc}
  $\eta_{0-}$ &  slope & region \\[3pt]
\hline\\[-9pt]
$-\frac{1}{3}(1-\sqrt{10})$  & $3 - \sqrt{10}$ & $\QW(0)$ \\[6pt]
$\frac{1}{3}$  & $-\frac{1}{2}$ & $\QW(0)$
\end{tabular}
\end{center}
The \fucik\ curves of $A_{5}$ emanating from $(0, 0)$ together with their tangent lines 
are visualized in Figure~\ref{fig_degenerate} (left). 
Similarly, we could obtain slopes of tangent lines of \fucik\ curves of 
the following matrix (visualized in Figure \ref{fig_degenerate} (right))
$$
  A_{6} = \left[
  \begin{array}{rrrr}
    -2 &  2 &  0 &  2 \\
     6 & -4 &  2 & -10 \\
    -2 &  2 &  0 &  6 \\
    -4 &  3 & -1 &  6
  \end{array}\right].
$$

\end{example}

\section{Generalization}
\label{sec_general}

The first possible generalization concerns the situation off the diagonal. Statements analogous to Theorems \ref{th_NC1}, \ref{th_NC2}, \ref{th_SC} can be formulated replacing $(\lambda,\lambda)$ by a \fucik\ eigenpair $(\alpha_0,\beta_0) \in \Sigma(A)$. However, due to the lack of linearity, the assumption of simplicity of $(\alpha_0,\beta_0)$ is now essential. Moreover, notice that for $A$ non-selfadjoint, the corresponding adjoint problem to \eqref{e.0} reads as
$$
A^{\ast} v = \tfrac{1}{2}\left(\alpha (I + \Xi(u)) + \beta(I - \Xi(u)) \right)v,
$$
where $u$ is the \fucik\ eigenfunction with all components nonzero corresponding to the simple \fucik\ eigenpair $(\alpha,\beta) \in \Sigma(A)$. That is, the adjoint problem is linear with coefficients driven by the nodal properties of the solution of the original (nonlinear) \fucik\ problem. Due to these technical obstacles we do not formulate and prove the assertions within this paper. The case of $A$ selfadjoint is treated in \cite{maroncelli}.

\bigskip

The second possible generalization concerns more general structures than matrices. In particular, let $H$ be a Hilbert space over the field of real numbers $\R$, endowed with a scalar product $\skals{\cdot}{\cdot}$ and a norm $\norma{\cdot}$ induced by the scalar product.
Moreover, let $H$ be ordered by a closed convex cone $K$ such that $(H, \skals{\cdot}{\cdot}, K)$ is a Hilbert lattice. For any element $u \in H$, we can then define its positive and negative parts by
$$
u^+ = 0 \vee u = \PK(u),\qquad
u^- = 0 \vee (-u) =  -\PP_{\rm -K}(u),
$$
where $\PK$ denotes the orthogonal projection onto $K$. 

Further, let $L: \dom(L) \subset H \to H$ be a linear closed operator
with $\dom(L)$ {dense in} $H$.
The \fucik~spectrum of  $L$ is again the set $\Sigma(L)$ of all pairs
$(\alpha, \beta)\in\R^{2}$, for which the problem $Lu = \alpha u^+ - \beta u^-$
has a~nontrivial solution~$u$.

If $\lambda \in \sigma(L) \cap \R$ is an isolated real eigenvalue of $L$ and 
$(L - \lambda I)$ is a Fredholm operator, i.e.,
$$
\dim \Ker (L - \lambda  I) < +\infty, \qquad \Img (L - \lambda  I) \mbox{~is~closed}, \qquad \codim \Img (L - \lambda I) < +\infty,
$$
then we can formulate statements analogous to Theorems \ref{th_NC1}, \ref{th_NC2}. 
If we in addition assume that the projection operator $\PK$ on the cone $K$ has 
a strong Fr\'echet derivative at the eigenvector $u_{0}$ corresponding to $\lambda$, 
we can formulate also an analogue of Theorem \ref{th_SC}. Again, due to technical 
difficulties and in an effort to preserve the clarity of the text, we do not present 
these results in detail within this paper. The case of $H = L^2(\Omega)$ with a cone 
$K = \{u \in H: ~u = u(x) \geq 0 ~~\mbox{a.e.~} x \in \Omega\}$
and $L$ being a real linear selfadjoint differential operator is treated in \cite{bfs1}, \cite{bfs2}.


\bigskip

\begin{acknowledgement}
	The first author was supported by the Grant Agency of the Czech Republic, Grant No. 22--18261S.
\end{acknowledgement}

\bibliographystyle{abbrv}
\bibliography{references}
\end{document}